\documentclass[9pt, a4paper]{article}
\usepackage{amscd,amsmath,amsfonts,amssymb,amsthm,cite,mathrsfs,color} 
\usepackage{dsfont} \usepackage{leftidx}
\usepackage{tikz} \usepackage{enumerate}
 \usepackage{appendix}

\theoremstyle{plain}
\newtheorem{theorem}{Theorem}[section]

\theoremstyle{definition}
\newtheorem{remark}[theorem]{Remark}
\theoremstyle{definition}

\theoremstyle{remark}
\newtheorem*{acknow}{Acknowledgments}

\numberwithin{equation}{section}
\numberwithin{theorem}{section}

\def\be{\begin{equation}}
\def\ee{\end{equation}}
\def\bae{\begin{eqnarray}}
\def\eae{\end{eqnarray}}

\def\LE{\mathrm{LE}}
\def\GE{\mathrm{GE}}

\begin{document}

\title{Phase transitions 
 for products  of   characteristic polynomials   under  Dyson  
  Brownian motion}
\author{Peter  J. Forrester\footnotemark[1]   ~ and  Dang-Zheng Liu\footnotemark[2]}
\maketitle
\renewcommand{\thefootnote}{\fnsymbol{footnote}}
\footnotetext[1]{ARC Centre of Excellence for Mathematical and Statistical Frontiers, School of Mathematics and Statistics, The University of Melbourne, Victoria 3010, Australia. Email: pjforr@unimelb.edu.au}
\footnotetext[2]{CAS Key Laboratory of Wu Wen-Tsun Mathematics, School of Mathematical Sciences, University of Science and Technology of China, Hefei 230026, P.R.~China. E-mail: dzliu\symbol{'100}ustc.edu.cn}



\begin{abstract}
We study the averaged products of characteristic polynomials  for   the Gaussian and Laguerre $\beta$-ensembles  with external source, and prove Pearcey-type phase transitions for particular full rank perturbations of source.
The phases are characterised by determining the explicit functional forms of the scaled limits of the averaged products of characteristic polynomials, which are given as certain multidimensional integrals, with dimension
equal to the number of products. 
\end{abstract}

\section{Introduction}

\subsection{Matrix-valued Brownian motion}
 
Let $Y$ be an $N \times N$ standard Gaussian matrix with real, complex or quaternion entries, the latter represented as particular $2 \times 2$ complex
matrices. Forming $G = {1 \over 2} (Y + Y^*)$ then gives a real symmetric, complex Hermitian or self-dual quaternion random matrix. This is the
construction of matrices  drawn from the Gaussian orthogonal, Gaussian unitary and Gaussian symplectic ensembles respectively (GOE, GUE and GSE) --- the
naming relates to the subset of unitary matrices which diagonalise $G$. Associated with each ensemble is a Dyson index $\beta$, which takes on the value
1,2, or 4 depending on the number of independent real and imaginary parts in a single entry of $G$.

The dynamical 
description of  the eigenvalues of GOE, GUE and GSE random matrices as diffusion processes was first stated by Dyson \cite{dy} in 1962.      
 Replacing the real and imaginary parts of each entry of $Y$ by independent Brownian motions, Dyson constructed a matrix-valued random process
 and further observed that the eigenvalues satisfy a system of stochastic differential equations (SDEs), specifying what is known as Dyson Brownian motion.
 In subsequent years, Dyson Brownian motion has been both an important research topic in its own right, and an effective tool
 in random matrix theory;  see e.g.~\cite{Ch91,RS93, Jo01, ADM, EY11,LSY19}. In particular, see Erd\"os and Yau's book \cite{EY17}  for  a comprehensive survey on   its application to
 the universality problem  for Wigner matrices. 
 
 To give some more detail, let $\{B_{j,k}(t), \tilde{B}_{j,k}(t): 1\leq j\leq k\leq N\}$ be  a set of i.i.d.~real-valued standard Brownian motions.
  The symmetric ($\beta = 1$) and complex Hermitian ($\beta = 2$) matrix-valued Brownian motion $H(t)$, is a random process on matrices  with
  entries equal to 
 \begin{equation}
 H_{j,k}=\begin{cases} 
  \frac{1}{\sqrt{2}}\big(B_{j,k}(t)+i(\beta-1) \tilde{B}_{j,k}(t)\big), & j<k, \\
  B_{j,j}, & j=k. \end{cases}
 \end{equation}
For $t\geq 0$,  let $x_{1}(t), \ldots, x_{N}(t)$ denote the eigenvalues of the Hermitian matrix
 \begin{equation}
 X(t)=H(t)+X(0)\label{GE} 
 \end{equation}
 where $X(0)$ is a fixed symmetric ($\beta = 1$) or complex Hermitian ($\beta = 2$) matrix. 
 Let $B_1(t), \dots$, $B_N(t)$ denote standard Brownian motions.
 A fundamental observation of Dyson \cite{dy} (see  \cite{AGZ} and \cite{Ka16} for text book treatments)
 is that the eigenvalues satisfy a system of SDEs 
 \begin{equation}
 dx_{i}(t)=dB_{i}(t)+\frac{\beta}{2}\sum_{j: j\neq i}\frac{1}{x_{i}(t)-x_{j}(t)}dt, \quad i=1, \ldots, N, \label{GEprocess}
 \end{equation}
with $\{x_{i}(0) \}_{i=1}^N$ the eigenvalues of $X(0)$.

Also of interest is the extension of (\ref{GE}) to chiral matrices specified by
\begin{equation}\label{1.4a}
X(t) = \begin{bmatrix} 0_{p \times p} & Z(t) \\
Z^*(t) & 0_{N \times N} \end{bmatrix} +  \begin{bmatrix} 0_{p \times p} & A \\
A^* & 0_{N \times N} \end{bmatrix}.
 \end{equation}
 Here $Z(t)$ is a $p \times N$ random matrix with entries i.i.d.~real ($\beta = 1$) or complex ($\beta = 2$) Brownian motions, while $A$ is a fixed
 $p \times N$ matrix with real ($\beta = 1$) or complex ($\beta = 2$) entries. For $p \ge N$, this matrix structure implies $X(t)$ has $p-N$ zero eigenvalues.
 The remaining $2N$ eigenvalues occur in $\pm$ pairs, with the positive eigenvalues equal to the eigenvalues of the matrix
 \begin{equation}\label{1.4b}
 (Z(t) + A)^* (Z(t) + A).
  \end{equation}
  It was shown by Bru \cite{Bru} in the real case, and by Konig and O'Connell \cite{KO} in the complex case that the eigenvalues of 
  (\ref{1.4b}) evolve as the system of SDEs
\begin{equation}
 dx_{i}(t)=2\sqrt{x_{i}(t)}dB_{i}(t)+\beta \bigg( p+\sum_{j: j\neq i}\frac{x_{i}(t)+x_{j}(t)}{x_{i}(t)-x_{j}(t)}\bigg)dt, \quad i=1, \ldots, N, \label{LEprocess}
 \end{equation}  
with $\{x_{i}(0) \}_{i=1}^N$ the eigenvalues of $A^* A$.

Dyson, in his selected papers \cite{dy96},   commented  on the matrix-valued Brownian motion: ``the physical motivation for introducing it is that it represents a system whose Hamiltonian is a sum of two parts, one known and one unknown".   For the known part, introduce the notation $x_i(0) = f_i$ ($i=1,\dots,N)$ for the eigenvalues of $X(0)$ in (\ref{GE}) and of $A^*A$ in (\ref{1.4b}), whereas for the eigenvalues for general $t$ in these matrices
introduce the notation $x_i(t) = x_i$ ($i=1,\dots,N)$. Let $\Delta_{N}(x):=\prod_{1\leq j<k\leq
N}(x_{k}-x_{j})$ denote the Vandermonde product. The eigenvalue probability density for the matrices (\ref{GE}) can be calculated as \cite[Eqs.~(11.101) and (13.146) with $t \to 1$, $\tilde{\beta} \mapsto 1/2T$ then
$T \mapsto \sqrt{t}$]{forrester}
\begin{multline}  \label{PDFforGauss}
P^{(G)}_{t,N}(x) =\frac{1}{\Gamma_{\beta,N}} t^{- \frac{1}{4}N(2+\beta(N-1))}
  \prod_{i=1}^{N}e^{-\frac{1}{2t}(x^{2}_i+f^{2}_i) }  \ |\Delta_{N}(x)|^{\beta} {{\phantom{k}}_0\mathcal{F}_0}^{\!(2/\beta)}\big(\frac{1}{\sqrt{t}}x;\frac{1}{\sqrt{t}}f\big),
\end{multline}
while for the functional form of the eigenvalue probability density for the matrices (\ref{1.4b}) we have \cite[Eqs.~(11.105) and (13.147) with $t \to 1$, $\tilde{\beta} \mapsto 1/T$ then
$T \mapsto t$ and the change of variables $x_i^2 \mapsto x_i$, $(x_i^{(0)})^2 \mapsto f_i$]{forrester}
\begin{multline}  \label{PDFforLaguerre}
P^{(L)}_{t,N}(x) =\frac{1}{Z_{a,\beta,N}}t^{-aN-\frac{1}{2}\beta N(N-1)}
  \prod_{i=1}^{N}x_{i}^{a-1}e^{-\frac{1}{t}(x_i+f_i)}  \\  \times  \ |\Delta_{N}(x)|^{\beta}\,  {{\phantom{k}}_0\mathcal{F}_1}^{\!(2/\beta)}\big(a+\frac{1}{2}\beta(N-1);\frac{1}{t}x;\frac{1}{t}f\big),
\end{multline}
where $a = {\beta \over 2} (p-N+1)$. In the present setting the multivariate functions $ {{\phantom{k}}_0\mathcal{F}_0}^{\!(2/\beta)}$, $ {{\phantom{k}}_0\mathcal{F}_1}^{\!(2/\beta)}$ are most
naturally defined as the matrix integrals
\begin{align*}
{{\phantom{k}}_0\mathcal{F}_0}^{\!(2/\beta)}\big(x;f\big) & =\int_{Q(N)} e^{\mathrm{tr}(Q\Lambda_{x}Q^{-1}\Lambda_{f})} \, dQ, \\
{{\phantom{k}}_1\mathcal{F}_0}^{\!(2/\beta)}\big(x;f\big) & =\int_{Q(N)} dQ  \int_{Q(p)} dR  \, e^{\mathrm{tr}(Q \tilde{\Lambda}_{x} R^{-1}\tilde{\Lambda}_{f}^{t} +  \tilde{\Lambda}_{f} R  \tilde{\Lambda}_{x}^{t}
Q^{-1})},
\end{align*}
with $ \tilde{\Lambda}_{x}=  \begin{bmatrix}  \sqrt{ \Lambda_x}&
 0_{N\times (p-N)} \end{bmatrix}$. 
Here $Q(N)$ denotes the classical groups $O(N)$ ($\beta = 1$), $U(N)$ ($\beta = 2$) and $Sp(2N)$ ($\beta = 4)$, and for $Q \in Q(N)$, $dQ$ denotes the corresponding normalized Haar measure.
The normalization constants $\Gamma_{\beta,N}$ and $Z_{a,\beta,N}$ are   given  in Appendix \ref{A.B}, and 
more detail about hypergeometric functions is given in Appendix \ref{A.A}.

The SDEs (\ref{GEprocess}) and (\ref{LEprocess}) have meaning for general $\beta > 0$, and can be shown to correspond to the Fokker-Planck dynamics of certain particle systems on the line
and half line respectively, which interact via a logarithmic potential in the presence of a heat bath at inverse temperature $\beta$; see e.g.~\cite[Ch.~11]{forrester}.
The particle probability density functions are again given by (\ref{PDFforGauss}) and (\ref{PDFforLaguerre}), but now with 
$ {{\phantom{k}}_0\mathcal{F}_0}^{\!(2/\beta)}$ and $ {{\phantom{k}}_0\mathcal{F}_1}^{\!(2/\beta)}$ defined as multivariate hypergeometric functions based on Jack polynomials;
see \cite[Ch.~13]{forrester} and Appendix \ref{A.B} for a brief summary. It is furthermore the case that $ {{\phantom{k}}_0\mathcal{F}_0}^{\!(2/\beta)}$ and $ {{\phantom{k}}_0\mathcal{F}_1}^{\!(2/\beta)}$ 
for general $\beta > 0$ correspond to the eigenvalue PDF of recursively defined random matrices \cite{for2012}, which in turn for $\beta = 1,2$ and 4 correspond to the
Gaussian and Laguerre ensembles with an external source; see  \cite{des,forrester,for2012,for2013}. In keeping with this, we will denote the ensembles corresponding to (\ref{PDFforGauss}) and
 (\ref{PDFforLaguerre}) as $\GE_{\beta, t, N}(x;f)$ and $\LE_{a,\beta,t, N}(x;f)$ --- in words Gaussian $\beta$-ensemble with a source, and Laguerre $\beta$-ensemble with a source ---
 respectively.

\subsection{Products of characteristic polynomials --- Riemann zeros}
The celebrated Riemann hypothesis in prime number theory asserts that the complex zeros of the Riemann zeta function
$$
\zeta (s) = \sum_{n=1}^\infty {1 \over n^s} = \displaystyle \prod_{\rm primes} \Big ( 1 - {1 \over p^s} \Big )^{-1}, \qquad {\rm Re} \, (s) > 1,
$$
when analytically continued in the whole complex plane, are all of the form $s = {1 \over 2} \pm i E$, $E > 0$. These are termed the Riemann zeros.
There is a conjecture attributed to Hilbert and P\'{o}lya \cite{wikiHP} asserting that the Riemann zeros correspond to the eigenvalues of an as yet unknown
unbounded self adjoint operator. From the topic of quantum chaos there is a prediction \cite{BGS84} that the highly excited energy levels of a generic Schr\"odinger operator have
the same statistics as the bulk eigenvalues of large real symmetric (assuming a time reversal symmetry), or complex Hermitian (no time reversal symmetry). In keeping with
these points, and based on analytic evidence from the work of Montgomery \cite{Mo73}, and large scale, high precision numerical work of Odlyzko \cite{Od89} of the
10${}^{20}$-th Riemann zero and over 70 million of its neighbours, the Montgomery--Odlyzko law asserts that the large Riemann zeros have the same statistical properties as the bulk
eigenvalues of large complex Hermitian matrices.

Keating and Snaith \cite{ks} extended the Montgomery--Odlyzko law by proposing the use of the characteristic polynomial of Haar distributed random unitary matrices
$Z(U,\theta) := \det (I - U e^{-i \theta})$ to model the statistical  properties of $\zeta(s)$ for $s = {1 \over 2} + i T$, $T \gg 1$. It has been known since the work of
Dyson \cite{Dy62} that the statistical properties of the eigenvalues of large Haar distributed unitary matrices coincide with the statistical properties of the bulk
eigenvalues of large GUE matrices. Specifically, with $N = \log T$ (the average spacing
between eigenvalues and Riemann zeros then agree to leading order) it was hypothesised that the statistical properties of $Z(U,\theta)$ correctly give the corresponding
statistical properties of $\zeta(0.5+iT)$ up to known arithmetic factors. A celebrated example exhibited in \cite{ks} applied the moments formula for $Z(U,\theta)$ \cite{BF97f,ks}
\begin{equation}\label{Em}
\Big \langle \, | Z(U,\theta) |^{2k}  \Big \rangle = \prod_{j=1}^N {\Gamma(j) \Gamma(j + 2k) \over \Gamma(j+k)^2 } \mathop{\sim}\limits_{N \to \infty} {G^2(k+1) \over G(2k+1)} N^{k^2},
\end{equation}
where $G(z)$ denotes the Barnes $G$-function, and satisfies $G(z+1) = \Gamma(z) G(z)$, to predict the corresponding asymptotic formula for the
moments of the Riemann zeta function
$$
{1 \over T} \int_0^T  \Big | \zeta \Big ( {1 \over 2} + i t \Big ) \Big |^{2k} \, dt  \mathop{\sim}\limits_{T \to \infty} {G^2(k+1) \over G(2k+1)} a(k) \Big ( \log {t \over 2 \pi} \Big )^{k^2},
$$
where $a(k)$ is a known number theoretic constant.

\subsection{Products of characteristic polynomials --- phase transition}

The general topic of averages of products and ratios of characteristic polynomials in random matrix ensembles has attracted an enormous amount of research; papers relevant to the
present study, where we focus attention on
\begin{equation} \label{productmeanGE}
K_{t, N}^{(G)}(s;f):=\left\langle \prod_{j=1}^n\prod_{k=1}^N (s_j-\sqrt{\frac{2}{\beta}}x_k)\right\rangle_{x\in\GE_{\beta,t,N}(x;f)}
\end{equation}
 and
\begin{equation} \label{productmeanLE}
K_{t,N}^{(L)}(s;f):=\left\langle \prod_{j=1}^n\prod_{k=1}^N (s_j-\frac{2}{\beta}x_k)\right\rangle_{x\in\LE_{a,\beta,t,N}(x;f)},
\end{equation}
 include \cite{af,bds,bs,bh,bh00,bh3,bh4,dl,dl2,hko01}. In the case $n=1$ and with $\beta = 2$ these averages are well known to relate to multiple orthogonal polynomials
 associated with the corresponding ensembles; see \cite{BK05,df1b}.
 
 We seek the asymptotic form of (\ref{productmeanGE}) and (\ref{productmeanLE}) in settings corresponding to what for $\beta = 2$ has been termed the 
 Pearcey universality class; see  \cite{av07,bk07,bh98, kfw11,tw2006}. In the Brownian motion picture, this corresponds to the circumstance when in the
 ensemble $\GE_{\beta,t,N}(x;f)$ the eigenvalues take on one of two values chosen symmetrically about the origin. As time increases, the eigenvalues
 spread, first in the neighourhood of the two values, then eventually colliding in the neighbourhood of the origin. It is literally a ``collision" with the origin that
 is the physical mechanism of the transition in the ensemble $\LE_{\beta,t,N}(x;f)$, when the initial condition has all but a finite number of eigenvalues concentrated at a point
 some distance away on the positive half axis.
 
 Earlier studies for general $\beta > 0$ have considered phase transitions in the finite rank case, say $f_{r+1}=\cdots=f_N=0$, specifying the
 scaling limits of the averages  (\ref{productmeanGE}) and (\ref{productmeanLE}); see 
  \cite{for2012} ($n=1$ and finite $r$), \cite{dl}($r=0$ and finite $n$) and \cite{dl2} (finite $n,r$). This follows works on the scaling limits of the distribution function
  of the largest eigenvalue  in  the cases $\beta = 1,2$ and 4 corresponding to a variant of the Laguerre $\beta$-ensemble with a source
  known as the general variance Wishart ensemble \cite{for2013,DEKV13}. The analogue of finite rank is then a spiking of the
  covariance matrix;  see \cite{bbp}  for spiked complex Wishart matrices   and  \cite{bv1,bv2,for2013,jkt,wang} for spiked real and quaternion  Wishart matrices.
  
  The starting point is to implement duality formulas for the averaged products of characteristic polynomials in the two ensembles  $\GE_{\beta,t,N}(x;f)$ and
  $\LE_{\beta,t,N}(x;f)$, due to Desrosiers \cite[Proposition 8]{des} (see also \cite{for2012}), which read
\begin{multline}  \label{dualityGE}
\left\langle \prod_{j=1}^n\prod_{k=1}^N (s_j-i\sqrt{\frac{2}{\beta}}x_k)\right\rangle_{x\in \GE_{\beta,t,N}(x;f)}=
\left\langle \prod_{k=1}^N \prod_{j=1}^n (x_j-i\sqrt{\frac{2}{\beta}}f_k)\right\rangle_{x\in \GE_{4/\beta,t,n}(x;s)},
\end{multline}  and
\begin{multline}   \label{dualityLE}
\left\langle \prod_{j=1}^n\prod_{k=1}^N (s_j+\frac{2}{\beta}x_k)\right\rangle_{x\in \LE_{a,\beta,t, N}(x;f)}=
\left\langle \prod_{k=1}^N \prod_{j=1}^n (x_j+\frac{2}{\beta}f_k)\right\rangle_{x\in \LE_{ 2a/\beta, 4/\beta,t,n}(x;s)},
\end{multline} 
where the variables $s_{j}$ denote the argument of  characteristic polynomials. We stress  that the duality relation has transformed the $N$-dimensional integral on the LHS into  an $n$-dimensional integral on the RHS, which  is suitable for     asymptotic analysis  as  $N\to\infty$ whenever $n$ is fixed, at least in  principle; see \cite{Fo92,dl,dl2} for some earlier relevant results.

The rest of the paper is organized as follows. Asymptotic analysis relating to the Pearcey universality class for the Gaussian $\beta$-ensemble with a source is undertaken in Section \ref{sect:gau},
and that for the Laguerre $\beta$-ensemble with a source in Section \ref{sect:lag}.
Appendix \ref{A.B} lists some constants,
and Appendix \ref{A.A} gives a brief account of the hypergeometric functions in  (\ref{PDFforGauss}) and (\ref{PDFforLaguerre}).

\section{Transition for the Gaussian ensemble}\label{sect:gau}

In this section we  study  scaled limits of the averaged product of characteristic polynomials for the Gaussian $\beta$-ensemble with a source, under the assumption that the vector $f$ of initial eigenvalues
takes on just two values chosen symmetrically about the origin, and in (close to) equal proportion.
  Rescaling  the time as $t=\hat{t}/N$, we will see that as $\hat{t}$ changes a critical value $t_{c}$ is reached, and so distinguishing three different  regimes:  (i) subcritical regime of $\hat{t}>t_{c}$; (ii) critical regime  of $\hat{t}=t_{c}$;   and  (iii) supercritical regime of $\hat{t}<t_c$.

Two families of multivariate functions of extended Selberg type are required.
One is  the  Pearcey weighted function  \begin{multline} P_{n,m}^{(\alpha)}(\tau;y;\sigma):= \frac{1}{\Gamma_{2/\alpha,n}} \int_{\mathbb{R}^n}e^{-\sum_{j=1}^{n}(\frac{1}{4}u_{j}^{4}+\frac{\tau}{2}u_j^{2})} \\
\times  \prod_{j=1}^{n}\prod_{k=1}^{m}(i u_{j}+\sigma_{k}) \,|\Delta_{n}(u)|^{\frac{2}{\alpha}} {\phantom{j}}_0\mathcal{F}_0^{(\alpha)}(iu;y)\, d^nu, \label{Pearceyfunctions}\end{multline}
where $\alpha\in \mathbb{R}_{+}, \tau\in \mathbb{R}$ and $y \in \mathbb{R}^n, \sigma \in \mathbb{R}^m$.  The other is   the   Gaussian weighted  function
 \begin{multline}\label{Gdef} G_{n,m}^{(\alpha)}(y;\sigma)=\frac{1}{\Gamma_{2/\alpha,n}} \int_{\mathbb{R}^n} e^{-\sum_{j=1}^{n}\frac{1}{2}(u_j^{2}-y_{j}^2)} \\
\times  \prod_{j=1}^{n}\prod_{k=1}^{m}(i u_{j}+\sigma_{k}) \,|\Delta_{n}(u)|^{\frac{2}{\alpha}} {\phantom{j}}_0\mathcal{F}_0^{(\alpha)}(iu;y)\, d^nu.  \end{multline}
 The latter has been defined in \cite{dl2}   except that   the trivial factor $e^{\sum_{j=1}^{n}y_j^{2}/2}$ was absent there. We  immediately  see from the duality relation \eqref{dualityGE} that
\be G_{n,m}^{(\alpha)}(y;\sigma)=(i\sqrt{\alpha})^{mn}G_{m,n}^{(1/\alpha)}(i\sqrt{\alpha}\sigma;i\sqrt{\alpha}y).\ee

\begin{theorem}  \label{zerotransition}
 For a fixed integer $r\geq 0$, suppose that $N$ is a positive integer   such that $N+r$ is even and moreover that
 \be f_{r+1}=\cdots=f_{(N+r)/2}=-f_{1+(N+r)/2}=\cdots=-f_{N}=\sqrt{\beta/2}\, b, \ b>0.\label{confluentcaseG}\ee
With   \eqref{productmeanGE}, as $N\rightarrow \infty$ the following hold uniformly for any   $y_1, \ldots, y_n$ in a compact subset of $\mathbb{R}$.

\begin{enumerate}[(i)]

\item  
When $t>b^2$, let  $f_1, \ldots, f_r$   be  in a compact subset of $\mathbb{R}$ and  set
     \begin{equation} s_j=\frac{t}{\sqrt{t-b^2}}\frac{y_j}{N}, \quad j=1,\ldots,n. \label{subscalingG}\end{equation}
  For even $n$ we have
   \begin{equation}\label{2.6}
     \lim_{N \to \infty} \frac{1}{\Psi_{\mathrm{sub}}}K_{t/N,N}^{(G)}(s;f)= \gamma_{n/2}(4/\beta)\,e^{-i \sum_{j=1}^{n}y_{j}}\! {{\phantom{k}}_1F_1}^{\!(\beta/2)}\big(n/\beta;2n/\beta;2i y\big),\end{equation}
where the function $_{1}F_{1}$   is defined in \eqref{hfpq} of  Appendix \ref{A.A}.
\item   
 When $t=b^{2}(1+\tau/\sqrt{N})^{-1}$, let  $f_{m+1}, \ldots, f_{r}$  be  in a compact subset of $\mathbb{R}\!\setminus\!\{0\}$ with $0\leq m\leq r$, and set
    \begin{equation} s_j= N^{-\frac{3}{4}} y_j, \ j=1,\ldots,n  \  \mbox{and} \
  f_k=\sqrt{\frac{\beta}{2}}  N^{-\frac{1}{4}} \sigma_k, \  k=1, \ldots, m   \label{criscalingG}\end{equation}
 with $\tau, \sigma_1, \ldots, \sigma_m \in \mathbb{R}$. We have
 \begin{equation}  \lim_{N \to \infty} \frac{1}{\Psi_{\mathrm{crit}}}K_{t/N,N}^{(G)}(s;f)=  P_{n,m}^{(\beta/2)}(\tau;y;\sigma).  \end{equation}
\item   
 When $t<b$,  let  $f_{m+1}, \ldots, f_{r}$ be  in a compact subset of $\mathbb{R}\!\setminus\!\{0\}$ with $1\leq m\leq r$,  and set
    \be s_j= \sqrt{\frac{t(b^2-t)}{b^2}}\frac{y_j}{\sqrt{N}}, j=1,\ldots,n,  
     \
  f_k=\sqrt{\frac{\beta tb^2}{2(b^2-t)}}\frac{ \sigma_k}{\sqrt{N}}, k=1, \ldots, m   \label{supscalingG}\ee
 with   $\sigma_1, \ldots, \sigma_m\in \mathbb{R}$. We have
 \begin{equation}    \lim_{N \to \infty} \frac{1}{\Psi_{\mathrm{sup}}}K_{t/N,N}^{(G)}(s;f) e^{\frac{t}{2b^2} \sum_{j=1}^n y_{j}^{2}}= G_{n,m}^{(\beta/2)}(y;\sigma).  \end{equation}
 \end{enumerate}
Here $\gamma_{n/2}(4/\beta)$, $\Psi_\mathrm{sub}, \Psi_\mathrm{cri}$  and $\Psi_\mathrm{sup}$ are constants given in  Appendix \ref{appendix}.
 \end{theorem}

 The proofs of Theorem \ref{zerotransition} and   Theorem  \ref{transition} below build on asymptotic   results for  integrals of Selberg type; see Corollary  3.11    and  Corollary  3.12 in \cite{dl}.    Our new finding is   to establish  two families of phase transitions at the origin by adjusting external  parameters properly  and then by doing  very subtle calculations.    In the unitary case of $\beta=2$,  one can exactly solve the Gaussian and Laguerre ensembles with external source by using the famous HCIZ formula  (see e.g. \cite{forrester}),  and so gets more delicate results such as correlation functions and the distribution of the largest eigenvalue,  see \cite{bh98,tw2006} and \cite{kfw11}.   However, 
to our knowledge,  for $\beta \neq 2$,  even in the orthogonal case of $\beta=1$,  the  Pearcey-type results 
have previously been investigated only
for a single characteristic polynomial.
 
\begin{proof}[Proof of Theorem \ref{zerotransition}]  Introduce scaled variables \begin{equation} s_j= v+ \frac{1}{\rho N} y_j, \quad j=1,\ldots,n,\label{generalscalingG} \end{equation}
where $v$ and $\rho$ are yet to be determined.
With the assumption \eqref{confluentcaseG} in mind, application of the duality formula  \eqref{dualityGE} shows
\begin{align} \label{scalingproductmeanG}
&K_{t/N,N}^{(G)}(s;f)=\frac{1}{\Gamma_{4/\beta,n}}(-i)^{nN}  (N/t)^{\frac{n}{2}+\frac{n(n-1)}{\beta}} e^{\frac{N}{2t}\sum_{j=1}^{n}(v+ \frac{1}{\rho N} y_j)^{2}} I_N ,
\end{align} 
where  \begin{align}
 I_{N}=  \int_{\mathbb{R}^n}  & e^{-N \sum_{j=1}^n p(x_j)} \prod_{j=1}^{n} \prod_{k=1}^r (x_j-i\sqrt{2/\beta}f_k) \prod_{j=1}^{n} (x_{j}^2+b^2)^{-r/2}  \nonumber\\
  &\times    \ |\Delta_{n}(x)|^{4/\beta}\! {{\phantom{k}}_0\mathcal{F}_0}^{\!(\beta/2)}(x/t;iy/\rho) \,d^n x.\label{IintegralG}
\end{align}
Here the exponent is given by
\be p(z)=\frac{1}{2t}\big(z^{2}-2ivz-t\log(z^2+b^2)\big). \label{pfunctionG}\ee

We apply now the method of steepest descent to analyse the large $N$ form of the integral.
From \eqref{pfunctionG}, the   saddle point equation  reads 
\be p'(z)=\frac{z-iv}{t}-\frac{z}{z^2+b^2}=0, \label{saddleeqnG}\ee
 which further reduces to the cubic equation 
 \begin{equation} \xi^{3}-v\xi^{2}-(b^2-t)\xi+vb^2=0,\label{cubiceqnG}\end{equation}
 where  $\xi=-iz$.
 The above equation has appeared in \cite{bk2004,abk2005} and the detailed analysis about  its solutions plays a central role in   establishing  the local eigenvalue  statistics  in the bulk and at the soft edge in the
 case $\beta = 2$, where there is a determinantal structure.

To proceed, we take $v=0$. Then, as the time $t$ changes, there is a critical value $t_{c}=b^2$ at which (\ref{saddleeqnG}) has a double root, so
distinguishing the three different  regimes:  (i) subcritical for $t>t_{c}$; critical for $t=t_{c}$; 
 and  (iii) supercritical for $t<t_c$.

 \textbf{Subcritical case}: $t>b^2$. In this case  the two solutions of \eqref{saddleeqnG} are  $z_{\pm}=\pm \sqrt{t-b^2}$, and moreover $p(z)$ attains the same minimum over the   real axis exactly at $z=z_{\pm}$.  Obviously, we have
 \be p(z_{\pm})=\frac{1}{2}-\frac{b^2}{2t}-\frac{1}{2}\log t, \qquad   p''(z_{\pm})=\frac{2}{t^2}(t-b^2)>0.\ee   Recall the definition of the scaled variables \eqref{subscalingG},  by \cite[Corollary 3.12]{dl} we conclude   that as $N\rightarrow \infty$   the major contribution to the integral  $I_N$ comes from the  case that $n/2$  variables lie in the   neighborhood of $z_+$, and the other $n/2$  in that of $z_{-}$. Taking $\rho=\sqrt{t-b^2}/t$,  simple calculations then show
 \begin{align} I_{N}&\sim \binom{n}{n/2} (\Gamma_{4/\beta,n/2})^2 e^{-\frac{1}{2}nN(1-\frac{b^2}{t}-\log t)} t^{-\frac{1}{2}rn}  \prod_{k=1}^r \big(-(t-b^2)-\frac{2}{\beta}f^{2}_k\big)^{\frac{n}{2}}
       \nonumber\\
 &   \times  \big(t/\sqrt{2N(t-b^2)} \big)^{n+\frac{n}{\beta}(n-2)}\,  {{\phantom{k}}_0\mathcal{F}_0}^{\!(\beta/2)}\big(1^{(\frac{n}{2})},(-1)^{(\frac{n}{2})}; i y\big).\end{align}
 Together with \eqref{scalingproductmeanG}, upon noting the following identity  (cf. \cite[Corollary 2.3]{dl})
 \be {{\phantom{k}}_0\mathcal{F}_0}^{\!(\beta/2)}\big(1^{(\frac{n}{2})},(-1)^{(\frac{n}{2})}; i y\big)=e^{-i\sum_{j=1}^{n}y_{j}}\! {{\phantom{k}}_1F_1}^{\!(\beta/2)}\big(n/\beta;2n/\beta;2iy\big),\ee 
 where  the function $_{1}F_{1}$   is defined in \eqref{hfpq} of  Appendix \ref{A.A},  we thus complete the subcritical case (i).

\textbf{Critical case}: $t=b^2(1+\tau/\sqrt{N})^{-1}$.  We recall   the definition of the scaled variables \eqref{criscalingG} and let $\rho=N^{-1/4}$.
Changing variables $x_j=N^{-1/4}u_j$ in \eqref{IintegralG}, and noting the power series expansion
\be p(x_j)=-\log b +\frac{1}{2N}(\tau u^{2}_j+\frac{1}{2}u_{j}^{4})+\mathcal{O}(N^{-\frac{3}{2}}),\ee
 a simple calculation  shows
  \begin{multline} I_{N} \sim b^{n(N-r)} N^{-\frac{1}{2\beta}(n-1)n-\frac{1}{4}(m+1)n}(-i)^{rn}\prod_{j=m+1}^{r}\Big(\sqrt{\frac{2}{\beta}}f_j\Big)^n \\
   \int_{\mathbb{R}^n} \prod_{j=1}^{n}\Big(e^{-\frac{\tau }{2}u^{2}_j-\frac{1}{4}u_{j}^4}\prod_{k=1}^m (iu_j+\sigma_k)\Big)
     \ |\Delta_{n}(u)|^{4/\beta} {{\phantom{k}}_0\mathcal{F}_0}^{\!(\beta/2)}(iu;y) \,d^nu.\end{multline}

Combining \eqref{scalingproductmeanG} and the definition \eqref{Pearceyfunctions},  one  obtains
 \begin{equation}   K_{t/N,N}^{(G)}(s,f) \sim  \Psi_\mathrm{cri}\,  P_{n,m}^{(\beta/2)}(\tau;y;\sigma), \end{equation}
where $\Psi_\mathrm{cri}$  is given in  Appendix \ref{A.B}.

\textbf{Supcritical case}: $t<b^2$. In this case, $p(z)$ attains its global minimum over the   real axis exactly at $z=0$.    Let  $\rho=b/\sqrt{Nt(b^2-t)}$ and   change variables
$x_j=b\sqrt{t}u_j/\sqrt{N(b^2-t)}$ in \eqref{IintegralG}. Noting  the Taylor expansion
\begin{equation} p(x_j)=-\log b +\frac{1}{2N}u^{2}_j+\mathcal{O}(N^{-\frac{3}{2}}),\end{equation}
and recalling the scaled variables in \eqref{supscalingG},  we get the leading contribution 
 \begin{multline} I_{N} \sim (-i)^{rn}b^{n(N-r)}\Big(b\sqrt{t}/\sqrt{N(b^2-t)} \Big)^{\frac{2}{\beta}(n-1)n+(m+1)n}\prod_{j=m+1}^{r}\Big(\sqrt{\frac{2}{\beta}}f_j\Big)^n \\
 \int_{\mathbb{R}^n}
    \prod_{j=1}^{n}\Big(e^{-\frac{1}{2}u^{2}_j}\prod_{k=1}^m (iu_j+\sigma_k)\Big)
     \ |\Delta_{n}(u)|^{4/\beta} {{\phantom{k}}_0\mathcal{F}_0}^{\!(\beta/2)}(iu;y) \,d^n u.\end{multline}
 Substituting the above into  \eqref{scalingproductmeanG}, we thus get
 \begin{align} K_{t/N,N}^{(G)}&(s,f) \sim  \Psi_\mathrm{sup}\,e^{-\frac{t}{2b^2}\sum_{j=1}^{n} y_j^{2}}\,  G_{n,m}^{(\beta/2)}(y;\sigma). \end{align}

The proof of the theorem is thus completed.
\end{proof}

 We conclude this section with   a few  remarks  about  the assumption  \eqref{confluentcaseG}:  (i)  
This (or a similar) initial condition has been used  to  investigate the  Pearcey  phenomenon in random matrices with source and   non-intersecting  Brownian motions, see e.g.    \cite{av07,bk07, bh98,bh00,tw2006}.
(ii)  The techniques introduced   in the proof of Theorem \ref{zerotransition}  are applicable to a wider class of initial conditions, say,  $f_{r+1}=\cdots=f_{r+N_1}=b_1$ and  $f_{r+N_1}=\cdots=f_N=b_2$ with $N_1/N\to c\in (0,1)$.  In this case  the key  equation satisfied by saddle points, like  \eqref{cubiceqnG},   will become  more complicated and   it is believed that there is no new interesting phenomenon,  so we  don't proceed  further.

\section{Transition for the Laguerre ensemble} \label{sect:lag}
In this section  we also assume that all but finitely many source  eigenvalues, say $r$,  are   the same. A hard edge phase transition can be described as follows as the time $t$ changes. We first need to define two families of multivariate functions of Selberg type, for $a>0$ one is defined to be  \begin{multline} B_{n,m}^{(a,\alpha)}(y;\sigma;\tau):= \frac{1}{Z_{a,2/\alpha,n}}\int_{\mathbb{R}_{+}}^{n}\prod_{i=1}^{n}\Big(u_{i}^{a-1}e^{-\tau u_i-\frac{1}{2}u_{i}^2}\prod_{j=1}^m (u_i+\sigma_j) \Big) \\
 \times  \ |\Delta_{n}(u)|^{\frac{2}{\alpha}}\,  {{\phantom{k}}_0\mathcal{F}_1}^{\!(\alpha)}\big(a+\frac{1}{\alpha}(n-1);u;-y\big)\, d^{n}u, \label{criticalfunctions}\end{multline}
 while the other reads   \begin{multline} W_{n,m}^{(a,\alpha)}(y;\sigma):= \frac{1}{Z_{a,2/\alpha,n}} \int_{\mathbb{R}_{+}}^{n}   
 \prod_{i=1}^{n}\Big(u_{i}^{a-1}e^{- u_i+y_i}\prod_{j=1}^m (u_i+\frac{1}{\alpha}\sigma_j) \Big) \\
 \times  \ |\Delta_{n}(u)|^{\frac{2}{\alpha}}\,  {{\phantom{k}}_0\mathcal{F}_1}^{\!(\alpha)}\big(a+\frac{1}{\alpha}(n-1);u;-y\big)\, d^{n}u
  \label{laguerrefunctions}\end{multline}
It immediately follows from the duality relation \eqref{dualityLE} that   the latter satisfies
\begin{equation} W_{n,m}^{(a,\alpha)}(y;\sigma)=\alpha^{-mn}W_{m,n}^{(a/\alpha,1/\alpha)}(-\sigma;-y).\end{equation}
\begin{theorem} [Hard edge phase transition]\label{transition}
 Suppose that for a fixed integer $r\geq 0$,
 \be f_{r+1}=\cdots=f_N=\beta b/2, \qquad b>0.\label{confluentcase}\ee
With   \eqref{productmeanLE}, as $N\rightarrow \infty$ the following hold for any   $y_1, \ldots, y_n$ in a compact subset of $[0,\infty)$.

\begin{enumerate}[(i)]

\item    \label{subcriticallimit} For $t>b$, let  $f_1, \ldots, f_r$  be  in a compact subset of $[0,\infty)$ and set
     \be s_j=\frac{t^2}{t-b}\frac{y_j}{N^2}, \qquad j=1,\ldots,n.\label{subscaling}\ee
     We have \begin{equation}    \lim_{N \to \infty}  \frac{1}{\Phi_\mathrm{sub}}K_{t/N,N}^{(L)}(s;f)= {{\phantom{k}}_0F_1}^{\!(\beta/2)}\big(2(a+n-1)/\beta;-y\big). \end{equation}
     
\item 
  \label{criticallimit} For $t=b(1- \tau/\sqrt{N})$, let   $f_{m+1}, \ldots, f_{r}>0$ with  $0\leq m\leq r$ and  set
    \begin{equation} 
    s_j= \frac{b y_j}{N^{3/2}},  j=1,\ldots,n   \ \mbox{and}  \
  f_k=\frac{\beta b\sigma_k}{2\sqrt{N}},  \quad  k=1, \ldots, m   \label{criscaling} 
  \end{equation}
 with $\tau\in \mathbb{R}$ and $\sigma_1, \ldots, \sigma_m\geq 0$. We have
 \begin{equation}    \lim_{N \to \infty} \frac{1}{\Phi_\mathrm{cri}} K_{t/N,N}^{(L)}(s;f)= B_{n,m}^{(2a/\beta,\beta/2)}(\tau;y;\sigma).  \end{equation}
  
\item 
    \label{supercriticallimit} For $t<b$,  let  $f_{m+1}, \ldots, f_{r}>0$ with $1\leq m\leq r$ and  set
    \begin{equation} s_j= \frac{t(b-t)}{b}\frac{ y_j}{N},  j=1,\ldots,n  \  \mbox{and} \
  f_k=\frac{bt}{b-t}\frac{\sigma_j}{N},   k=1, \ldots, m   \label{supscaling}\end{equation}
 with   $\sigma_1, \ldots, \sigma_m\geq 0$. We have
 \begin{equation}   \lim_{N \to \infty}  \frac{1}{\Phi_\mathrm{sup}} K_{t/N,N}^{(L)}(s;f) e^{(t/b)\sum_{i=1}^n y_i}= \, W_{n,m}^{(2a/\beta,\beta/2)}(y;\sigma).  \end{equation}
\end{enumerate}
Here $\Phi_\mathrm{sub}, \Phi_\mathrm{cri}$  and $\Phi_\mathrm{sup}$ are constants given in  Appendix \ref{appendix}.
 \end{theorem}

\begin{proof} Recalling the assumption \eqref{confluentcase}, by the duality formula  \eqref{dualityLE}  we find 
\begin{align} \label{scalingproductmean}
&K_{t/N,N}^{(L)}(s;f)=\frac{1}{Z_{2a/\beta,4/\beta,n}}(-1)^{nN}  (N/t)^{2n(a+ n-1)/\beta} e^{(N/t)\sum_{i=1}^{n}s_i} I_N,
\end{align}
where  \begin{align}
 I_{N}=  \int_{\mathbb{R}_{+}^n} & e^{-N \sum_{i=1}^n p(x_i)} \prod_{i=1}^{n}\Big(x_{i}^{\frac{2a}{\beta}-1}(x_i+b)^{-r}\prod_{j=1}^r (x_i+\frac{2}{\beta}f_j)\Big)  \nonumber\\
  &\times  \ |\Delta_{n}(x)|^{4/\beta}\,  {{\phantom{k}}_0\mathcal{F}_1}^{\!(\beta/2)}(c; Nx/t;-Ns/t)  \,d^nx. \label{Iintegral}
\end{align}
Here $c:=2(a+n-1)/\beta$ and
\begin{equation} p(z)=\frac{z}{t}-\log(z+b). \label{pfunction}\end{equation}

Applying now the method of steepest descent, we see from \eqref{pfunction}  that   the   saddle point equation  reads 
\be p'(z)=\frac{1}{t}- \frac{1}{z+b}=0,\ee
 and further implies the saddle point $z_0=t-b$.  Next, we establish the hard-edge limits in three different  cases:   (i) subcritical regime  of $t>b>0$, (ii) critical regime of $t=b>0$  and (iii) supercritical regime of $0<t<b$.

 \textbf{Subcritical case}: $z_0>0$. In this case $p(z_0)=1-bt^{-1}-\log t$ and $p''(z_0)=t^{-2}>0$. Moreover, $p(z)$ attains the minimum over the positive real axis at $z=z_0$.   Recalling the scaled variables in \eqref{subscaling},  by \cite[Corollary 3.11]{dl} we conclude   that as $N\rightarrow \infty$ the leading term of the integral $I_N$ comes from the neighborhood of $z_0$.  Noting the change of variables \eqref{subscaling}, we have
 \begin{align} I_{N}&\sim  
\int_{\mathbb{R}^n}  e^{-nNp(z_0)-N \frac{p''(z_0)}{2}\sum_{i=1}^n  (x_i-z_0)^{2}} \big(z_{0}^{\frac{2a}{\beta}-1}(z_0+b)^{-r}\big)^{n}    \nonumber\\
 & \quad \times \prod_{j=1}^r (z_0+\frac{2}{\beta}f_j)^{n}\,  \ |\Delta_{n}(x)|^{4/\beta}   {{\phantom{k}}_0F_1}^{\!(\beta/2)}\big(c; -\frac{z_0}{t-b} y\big)  \,d^nx\\
 &=e^{-nN(1-\frac{b}{t}-\log t)}(t-b)^{(\frac{2a}{\beta}-1)n}t^{-rn} \prod_{j=1}^r (t-b+\frac{2}{\beta}f_j)^{n} \nonumber\\
 &\quad \times (t/\sqrt{N})^{\frac{2}{\beta}n(n-1)+n}  \,\Gamma_{4/\beta,n}\! {{\phantom{k}}_0F_1}^{\!(\beta/2)}(c; -y),\end{align}
 where the constant $\Gamma_{4/\beta,n}$ is given in    \eqref{constgamma}.

 Together with \eqref{scalingproductmean} we thus get

 \be \label{limit0F1} K_{t/N,N}^{(L)}(s,f) \sim \Phi_\mathrm{sub}\! {{\phantom{k}}_0F_1}^{\!(\beta/2)}(c; -y). \ee

\textbf{Critical case}: $z_0=0$.  Recall  the double scaling $t=b(1- \tau/\sqrt{N})$ and the scaled variables in \eqref{criscaling}.    After the change of variables $x_i=u_i /\sqrt{N}$ in \eqref{Iintegral}, noting the Taylor expansion
\be p(x_i)=-\log b +\frac{1}{N}(\tau u_i+\frac{1}{2}u_{i}^{2})+\mathcal{O}(N^{-\frac{3}{2}}),\ee
 a simple calculation   shows
  \begin{multline} I_{N} \sim b^{n(N-r)}\big(\frac{1}{\sqrt{N}})^{\frac{2}{\beta}(a+n-1)n+mn}\prod_{j=m+1}^{r}(\frac{2}{\beta}f_j)^n\, \int_{\mathbb{R}_{+}^n}  {{\phantom{k}}_0\mathcal{F}_1}^{\!(\beta/2)}(c;u;-y)  \\
   \times  \prod_{i=1}^{n}\Big(u_{i}^{\frac{2a}{\beta}-1}e^{-\tau u_i-\frac{1}{2}u_{i}^2}\prod_{j=1}^m (u_i+\sigma_j)\Big)
     \ |\Delta_{n}(u)|^{4/\beta}   d^n u.\end{multline}

Combining \eqref{scalingproductmean} and \eqref{criticalfunctions},  we obtain
 \begin{equation}   K_{t/N,N}^{(L)}(s,f) \sim  \Phi_\mathrm{cri}\,  B_{n,m}^{(2a/\beta,\beta/2)}(\tau;y;\sigma).\end{equation}

\textbf{Supercritical case}: $z_0<0$. In this case $p(z)$ attains its minimum over the positive real axis at $z=0$.     Change variables $x_i= btu_i/(N(b-t))$ in \eqref{Iintegral} and note  the Taylor expansion
\be p(x_i)=-\log b +N^{-1}u_i+\mathcal{O}(N^{-2}).\ee
Recalling now the scaled variables in \eqref{supscaling},  we have
 \begin{multline} I_{N} \sim b^{n(N-r)}\Big(\frac{bt}{N(b-t)}\Big)^{\frac{2}{\beta}(a+n-1)n+mn}\prod_{j=m+1}^{r}\big(\frac{2}{\beta}f_j\big)^n\,  \int_{\mathbb{R}_{+}^n} \prod_{i=1}^{n} u_{i}^{\frac{2a}{\beta}-1}   \\
   \quad \times  e^{- \sum_{i=1}^n u_i}\prod_{i=1}^{n} \prod_{j=1}^m \big(u_i+\frac{2}{\beta}\sigma_j\big)
     \ |\Delta_{n}(u)|^{4/\beta} {{\phantom{k}}_0\mathcal{F}_1}^{\!(\beta/2)}(c;u;-y)\,  d^{n}u.\end{multline}

 Substituting the above into  \eqref{scalingproductmean}, we thus get
 \begin{align} K_{\beta,N}^{(L)}&(s,f;t/N) \sim (-1)^{nN} b^{n(N-r)}\Big(\frac{b}{b-t}\Big)^{\frac{2}{\beta}(a+n-1)n+mn}\Big(\frac{t}{N}\Big)^{mn} \prod_{j=m+1}^{r}\big(\frac{2}{\beta}f_j\big)^n \nonumber\\
 & \times \frac{1}{Z_{2a/\beta,2/\beta,n}} e^{(1-\frac{t}{b})\sum_{i=1}^n y_i}\int_{0}^{\infty}\cdots
\int_{0}^{\infty}  \prod_{i=1}^{n}\Big(u_{i}^{\frac{2a}{\beta}-1}e^{-  u_i}\prod_{j=1}^m (u_i+\frac{2}{\beta}\sigma_j)\Big) \nonumber\\
    & \times
        \ |\Delta_{n}(u)|^{4/\beta} {{\phantom{k}}_0\mathcal{F}_1}^{\!(\beta/2)}(c;u;-y)\,  du_1 \cdots du_n.\end{align}
  From this, together with the duality formula \eqref{dualityLE} and the definition \eqref{laguerrefunctions},  we have the desired  result
     \begin{equation}   K_{t/N,N}^{(L)}(s,f) \sim   \Phi_\mathrm{sup} \  e^{-(t/b)\sum_{i=1}^n y_i}\, W_{n,m}^{(2a/\beta,\beta/2)}(y;\sigma).\end{equation}
\end{proof}
\begin{remark} Theorem \ref{transition} (i) also holds even when   $b=0$,  which can be proved  by following almost the same procedure.

\end{remark}

 \section{Appendices} 
  
         \subsection{Constants} \label{A.B}
       \label{appendix}

The normalization constants in the Gaussian and Laguerre $\beta$-ensembles with a source, being independent of the source and $t$, are the normalizations
corresponding to the case $f = (0,\dots,0)$ and $t=1$. The generalized hypergeometric functions are then equal to unity, and we can
read off from \cite[Eq.~(1.160)]{forrester} and  \cite[eqn (3.128)]{forrester} that
\be \label{constgamma}
 \Gamma_{\beta,n}=(2\pi)^{n/2}\prod_{j=1}^n\frac{\Gamma(1+j\beta/2)}{\Gamma(1+\beta/2)}.
\ee
and
 \be Z_{a,\beta,N}=\prod_{j=0}^{N-1}\frac{\Gamma(1+(1+j)\beta/2)\Gamma(a+j\beta/2)}{\Gamma(1+\beta/2)} \label{Lconst}.\ee
 A systematic way to obtain these evaluations is to use Selberg integral theory; see \cite[\S 4.7]{forrester}.

The constants associated with the Gaussian $\beta$-ensemble with a source appearing in
Theorem \ref{zerotransition} are
\begin{align}  \Psi_\mathrm{sub}&=    (-i)^{n(N-r)}\big(2(t-b^2)\big) ^{-\frac{n}{2}-\frac{n}{\beta}(\frac{n}{2}-1)-\frac{n}{2}} t^{-\frac{n}{\beta}-\frac{1}{2}(r-1)n}
       \nonumber\\
 &   \quad \times   e^{-\frac{1}{2}nN(1-\frac{b^2}{t}-\log t)} N^{\frac{n^2}{2\beta}} \prod_{j=1}^r \big(t-b^2+\frac{2}{\beta}f^{2}_j\big)^{\frac{n}{2}},
 \end{align}
\begin{align} 
 \gamma_{m}(\beta'\!)=\binom{2m}{m}\prod_{j=1}^m\frac{\Gamma(1+\beta'\!j/2)}{\Gamma(1+\beta'\!(m+j)/2)}
 \end{align}
\begin{align}  \Psi_\mathrm{crt} =  (-i)^{n(r+N)}b^{n(N-r)}N^{\frac{1}{4}(1-m)n+\frac{1}{2\beta}(n-1)n}\prod_{j=m+1}^{r}\big(\sqrt{\frac{2}{\beta}}f_j\big)^n,\end{align}
 and     \begin{align}  \Psi_\mathrm{sup} =      (-i)^{n(N+r)}b^{n(N-r)}\big(b/\sqrt{b^2-t} \big)^{\frac{2}{\beta}(n-1)n+(m+1)n} \nonumber\\ \times \big(t/N \big)^{\frac{mn}{2}}
  \prod_{j=m+1}^{r}\big(\sqrt{\frac{2}{\beta}}f_j\big)^n.\end{align}
The constants  associated with the Laguerre $\beta$-ensemble with a source appearing in Theorem \ref{transition} are
 \begin{multline} \Phi_\mathrm{sub}=\frac{\Gamma_{4/\beta,n}}{Z_{2a/\beta,4/\beta,n}}  \Big(1-\frac{b}{t}\Big)^{(\frac{2a}{\beta}-1)n}   \prod_{j=1}^r \big(1-\frac{b}{t}+\frac{2}{t\beta}f_j\big)^{n}\\
\times (-1)^{nN} e^{-nN(1-\frac{b}{t}-\log t)} N^{\frac{n(n-1)}{\beta}+(\frac{2a}{\beta}-\frac{1}{2})n},\end{multline}
\be \Phi_\mathrm{cri}=(-1)^{nN}b^{n(N-r)} b^{-\frac{2n}{\beta}(a+n-1)} N^{\frac{(a+n-1)n}{\beta}-\frac{nm}{2}} \prod_{j=m+1}^{r}(\frac{2}{\beta}f_j)^n,\ee
and
\be \Phi_\mathrm{sup}=(-1)^{nN}b^{n(N-r)}  \Big(\frac{b}{b-t}\Big)^{\frac{2}{\beta}(a+n-1)n+mn} \Big(\frac{t}{N}\Big)^{mn} \prod_{j=m+1}^{r}\big(\frac{2}{\beta}f_j\big)^n.\ee

    
      \subsection{Jack polynomials and hypergeometric functions} 
\label{hypergeometric}\label{A.A}
This appendix    provides a brief review of  Jack polynomials  and   hypergeometric functions \cite{stanley,yan}; see also \cite[Chapter 12]{forrester}.

A partition $\kappa = (\kappa_1, \kappa_2,\ldots,\kappa_i,\ldots,)$ is a sequence of non-negative integers $\kappa_i$ such that
\begin{equation*}
    \kappa_1\geq\kappa_2\geq\cdots\geq\kappa_i\geq\cdots
\end{equation*}
and only a finite number of the terms $\kappa_i$ are non-zero. The number of non-zero terms is referred to as the length of $\kappa$, and is denoted $\ell(\kappa)$. We shall not distinguish between two partitions that differ only by a string of zeros. The weight of a partition $\kappa$ is the sum
\begin{equation*}
    |\kappa|:= \kappa_1+\kappa_2+\cdots
\end{equation*}
of its parts, and its diagram is the set of points $(i,j)\in\mathbb{N}^2$ such that $1\leq j\leq\kappa_i$.
Reflection in the diagonal produces the conjugate partition
$\kappa^\prime=(\kappa_1',\kappa_2',\ldots)$.
The set of all partitions of a given weight is  partially ordered
by the dominance order: $\kappa\leq \sigma $ if and only if $\sum_{i=1}^k\kappa_i\leq \sum_{i=1}^k \sigma_i$ for all $k$.

Let $\Lambda_n(x)$ be the algebra of symmetric polynomials in $n$ variables $x_1,\ldots,x_n$ with coefficients in the field $\mathbb{F}=\mathbb{Q}(\alpha)$, which is  the field of  rational functions in the parameter $\alpha>0$.
It is invariant under the action of  homogeneous differential operators related to the Calogero-Sutherland models \cite{bf}:
$$ E_k=\sum_{i=1}^n x_i^k\frac{\partial}{\partial x_i}, \ D_k=\sum_{i=1}^n x_i^k\frac{\partial^2}{\partial x_i^2}+\frac{2}{\alpha}\sum_{1\leq i\neq j \leq n}\frac{x_i^k}{x_i-x_j}\frac{\partial}{\partial x_i},\  k=0,1,2,\ldots .
$$
The operators $E_1$ and $D_2$   can be used to define the Jack polynomials.  Indeed, for each partition $\kappa$, there exists a unique symmetric polynomial $P^{(\alpha)}_\kappa(x)$ that satisfies the following two  conditions \cite{stanley}:
\begin{align} \label{Jacktriang}(1)\qquad&P^{(\alpha)}_\kappa(x)=m_\kappa(x)+\sum_{\mu<\kappa}c_{\kappa\mu}m_\mu(x)&\text{(triangularity)}\\
\label{Jackeigen}(2)\qquad &\Big(D_2-\frac{2}{\alpha}(n-1)E_1\Big)P^{(\alpha)}_\kappa(x)=\epsilon_\kappa P^{(\alpha)}_\kappa(x)&\text{(eigenfunction)}\end{align}
where  $\epsilon_\kappa, c_{\kappa\mu} \in \mathbb{F}$.  Because of the tiangularity condition, $\Lambda_n(x)$ is also equal to the span over $\mathbb{F}$ of all Jack polynomials  $P^{(\alpha)}_\kappa(x)$, with $|\kappa|\leq n$.

For    $(i,j) \in \kappa$, let
$
    a_\kappa(i,j) = \kappa_i-j$ and $l_\kappa(i,j) = \kappa^\prime_j-i$. Introduce the hook-length of $\kappa$    defined by
\begin{equation*} \label{defhook}
    h_{\kappa}^{(\alpha)}=\prod_{(i,j)\in\kappa}\Big(1+a_\kappa(i,j)+\tfrac{1}{\alpha}l_\kappa(i,j)\Big),
\end{equation*}
and  the   $\alpha$-deformation of the Pochhammer symbol by
\begin{equation*}\label{defpochhammer}
    [x]^{(\alpha)}_\kappa = \prod_{1\leq i\leq \ell(\kappa)}\Big(x-\tfrac{i-1}{\alpha}\Big)_{\kappa_i}.
\end{equation*}
Here  $(x)_j\equiv x(x+1)\cdots(x+j-1)$. We   now turn  to   the  precise definition of the hypergeometric series associated with Jack polynomials, see e.g. \cite{yan}.
Given  $p,q\in\mathbb{N}_0=\{0, 1, 2,\ldots\}$,  let $a_1,\ldots,a_p, b_1,\ldots,b_q$ be complex numbers such that $(i-1)/\alpha-b_j\notin\mathbb{N}_0$ for all $i\in\mathbb{N}_0$.
The $(p,q)$-type hypergeometric series in two sets of $n$ variables  $x=(x_1,\ldots,x_n)$ and  $y=(y_1,\ldots,y_n)$  is defined as 
\begin{align}\label{hfpq2}
    {\phantom{j}}_p\mathcal{F}^{(\alpha)}_{q}&(a_1,\ldots,a_p;b_1,\ldots,b_q;x;y)
                =  \sum_{k=0}^{\infty}\, \sum_{\ell(\kappa)\leq n,|\kappa|=k} \frac{1}{h_{\kappa}^{(\alpha)}}\frac{\lbrack a_1\rbrack^{(\alpha)}_\kappa\cdots\lbrack a_p\rbrack^{(\alpha)}_\kappa}{\lbrack b_1\rbrack^{(\alpha)}_\kappa
    \cdots\lbrack b_q\rbrack^{(\alpha)}_\kappa}\frac{P_{\kappa}^{(\alpha)}(x)P_{\kappa}^{(\alpha)}(y)}{P_{\kappa}^{(\alpha)}(1^{(n)})},
\end{align}
where  the shorthand notation $1^{(n)}$ stands for $ 1,\ldots,1 $  with  $n$ times.
In particular, when   $y_1=\cdots=y_n=1$, it reduces to the hypergeometric series
\begin{equation}\label{hfpq}
  {\phantom{j}}_p F^{(\alpha)}_{q}(a_1,\ldots,a_p;b_1,\ldots,b_q;x) = \sum_{k=0}^{\infty}\sum_{|\kappa|=k} \frac{1}{h_{\kappa}^{(\alpha)}}\frac{\lbrack a_1\rbrack^{(\alpha)}_\kappa\cdots\lbrack a_p\rbrack^{(\alpha)}_\kappa}{\lbrack b_1\rbrack^{(\alpha)}_\kappa
    \cdots\lbrack b_q\rbrack^{(\alpha)}_\kappa}P_{\kappa}^{(\alpha)}(x). 
\end{equation}
 Note that when $p\leq q$, the   series \eqref{hfpq}  converges absolutely for all $x\in\mathbb{C}^n$.
 
 Setting $x = 0^{(n)}$ in either (\ref{hfpq2}) we see that the only term contributing is $\kappa$ having all
 parts equal to zero, which shows
\begin{align}\label{hfpq4}
    {\phantom{j}}_p\mathcal{F}^{(\alpha)}_{q}(a_1,\ldots,a_p;b_1,\ldots,b_q;0^{(n)};y)
                =  1.
                \end{align}
                Less immediate but also of interest is the simplification of (\ref{hfpq2}) in the case $y = y^{(n)}$.
                Thus (see e.g.~\cite[Eq.~(13.63)]{forrester})
\begin{align}\label{hfpq5}
  {\phantom{j}}_0\mathcal{F}^{(\alpha)}_{0}(x; y^{(n)}) =   {\phantom{j}}_0 F^{(\alpha)}_{0}(yx_1,\dots,yx_n) = \exp\! \Big ( y \sum_{i=1}^n x_i \Big ).
  \end{align}

\begin{acknow}
 
The authors were   supported by Australian Research Council 
(Grant No. DP170102028), 
  the National Natural Science Foundation of China   (Grant No. 11771417) and the Youth Innovation Promotion Association CAS   (Grant No. 2017491)   
\end{acknow}

\end{document}